\documentclass[11pt,a4paper]{article}
\usepackage[english]{babel}
\usepackage{amsmath,amsfonts,amsthm,amssymb,amsbsy,upref,color,graphicx,amscd,hyperref,makeidx,enumerate,bbm,comment}
\usepackage[a4paper,margin=2.5truecm]{geometry}
\usepackage[active]{srcltx}
\usepackage[latin1]{inputenc}

\DeclareMathOperator{\spt}{spt}

\def\A{\mathcal{A}}

\def\V{\mathcal{V}}

\def\eps{{\varepsilon}}
\def\O{\Omega}
\def\N{\mathbb{N}}
\def\R{\mathbb{R}}

\newcommand{\be}{\begin{equation}}
\newcommand{\ee}{\end{equation}}
\newcommand{\bib}[4]{\bibitem{#1}{\sc#2: }{\it#3. }{#4.}}

\newcommand{\sdi}{\sigma_{dis}}
\newcommand{\sdim}{\sigma^-_{dis}}

\numberwithin{equation}{section} \theoremstyle{plain}
\newtheorem{theo}{Theorem}[section]
\newtheorem{lemm}[theo]{Lemma}

\theoremstyle{remark}
\newtheorem{rema}[theo]{Remark}
\newtheorem{exam}[theo]{Example}

\theoremstyle{figure}

\newenvironment{ack}{{\bf Acknowledgements. }}

\title{Optimal design problems for Schr\"odinger operators\\ with noncompact resolvents}

\author{Guy Bouchitt\'e, Giuseppe Buttazzo}

\date{29 January 2015}

\begin{document}
\maketitle

\begin{abstract}
We consider optimization problems for cost functionals which depend on the negative spectrum of Schr\"odinger operators of the form $-\Delta+V(x)$, where $V$ is a potential, with prescribed compact support, which has to be determined. Under suitable assumptions the existence of an optimal potential is shown. This can be applied to 
interesting cases such as costs functions involving finitely many negative eigenvalues.
\end{abstract}



\section{Introduction}\label{sintro}

Optimization problems for spectral functionals are widely studied in the literature; in a general framework one may consider an admissible class $\A$ of operators and the problem is then formulated as
\be\label{genpb}
\min\big\{F\big(\sigma(A)\big)\ :\ A\in\A\big\}
\ee
where $\sigma(A)$ denotes the spectrum of the operator $A\in\A$ and $F$ is a suitable given cost function that depends on $\sigma(A)$.

The most studied case is when the admissible class $\A$ of operators consists of the Laplace operator $-\Delta$ over a variable domain $\O$, with homogeneous Dirichlet boundary conditions on $\partial\O$. If the Lebesgue measure $|\O|$ is supposed finite, the resolvent operators are compact and then their spectrum reduces to an increasing sequence of positive eigenvalues, so than the optimization problem \eqref{genpb} takes the form
$$\min\big\{F\big(\lambda_1(\O),\lambda_2(\O),\dots\big)\ :\ \O\in\O\big\}$$
where $\O$ indicates the class of admissible domains. We refer to \cite{bb05}, \cite{bremc} and to the references therein for a survey on this topic and for the various existence results that are available in this situation.

Optimization problems of the form \eqref{genpb} have been also considered in \cite{jep} for operators of Schr\"odinger type $-\Delta+V(x)$, under the assumption $V\ge0$ and on a fixed bounded domain, on the boundary of which the homogeneous Dirichlet conditions are imposed. Again, the resolvent operators are compact, hence their spectra are discrete and the optimization problem \eqref{genpb} takes the form
$$\min\big\{F\big(\lambda_1(V),\lambda_2(V),\dots\big)\ :\ V\in\V\big\}$$
where $\V$ indicates now the class of admissible potentials. Several existence results for optimal potentials have been obtained in \cite{jep} in this situation.

In the present paper we consider Schr\"odinger operators $-\Delta+V(x)$ where the potential $V$ is assumed to be compactly supported and is allowed to become negative. Thus the resolvent operators are not any more compact and the spectrum $\sigma(V)$, besides its continuous part, exhibits discrete negative eigenvalues. Such a situation occurs for instance in the context of very thin quantum waveguides, where a one-dimensional effective potential, depending on local curvature and twist, appears explicitly (see for instance \cite{CDFK}, \cite{BMT}, \cite{BMT2}). The optimization problems we consider are described in Theorems \ref{exth} and \ref{exth2}, in which we show the existence of optimal potentials.

Some examples illustrate the range of possibilities in which our existence results apply.

Along all the paper, the notation of function spaces $L^2$, $H^1$ and similar, without the indication of the domain of definition, is used when the domain is the whole $\R^d$. Similarly, the absence of the domain of integration in an integral means that the integral is made on the whole $\R^d$.

\section{Presentation of the problem}\label{spres}

The problems we aim to consider are of the form
\be\label{prob}
\min\big\{F(V)\ :\ V\in\A\big\}
\ee
where $F$ is a suitable cost functional and $\A$ is a suitable class of admissible potentials defined on $\R^d$. In order to simplify the presentation, we assume that all the potentials we consider have a support contained in a given compact set $K$. The admissible potentials may change sign and indeed their negative parts are mostly important for our purposes; the class $\A$ is then defined as
$$\A=\big\{V:\R^d\to\overline\R,\ \spt V\subset K\big\}.$$

For a Schr\"odinger operator $-\Delta+V(x)$ we denote by $\sigma(V)$ its spectrum and by $\sdi(V)$ its discrete part, consisting of {\it isolated} eigenvalues; finally $\sdim(V)$ will denote the part of $\sdi(V)$ which consists of {\it strictly negative} eigenvalues. By the Cwikel-Lieb-Rosenbljum bound (see for instance \cite{resi}) it is known that
\be\label{CLR}
\#\sdim(V)\le C_{q,d}\int|V^-|^q\,dx\qquad\forall d\ge3,\ \forall q\ge d/2.
\ee
where the eigenvalues are counted with their multiplicity. Other important inequalities we will use are the Lieb-Thirring inequality (see for instance \cite{lawe00,lieb})
which is valid in any dimension $d$.
\be\label{liethi}
\sum_{\lambda\in\sdim(V)}|\lambda|^{p-d/2}\le L_{p,d}\int|V^-|^p\,dx\qquad\forall p>d/2.
\ee
and the Keller inequality (see for instance \cite{CFL})
\be\label{Keller}
|\lambda_1|^{p-d/2}\le K_{p,d}\int|V^-|^p\,dx\qquad\forall p>d/2.
\ee

The cost functionals we consider are of the following two classes: 
\be\label{cost}
F(V)=\sum_{\lambda\in\sdim(V)} m_V(\lambda)\, h(\lambda)+k\int|V|^p\,dx
\ee
\be\label{cost2}
F(V)=g\big(\Phi(\sdim(V))\big)+k\int|V|^p\,dx \ .
\ee 
In definition \eqref{cost}, $k$ is a given positive number, $p>d/2$ and $m_V(\lambda)$ denotes the multiplicity of the eigenvalue $\lambda$. The function $h:\R\to]-\infty,+\infty]$ is a given lower semicontinuous function.

In definition \eqref{cost2}, we denoted by $\Phi$ the map which sends $\sdim(V)$ into the space $c_0(\R^-)$ of vanishing sequences of negative real numbers, defined as follows: let $\lambda_1\ge\lambda_2\ge\lambda_3\dots$ be an enumeration of the elements of $\sdim(V)$ in increasing order and repeated according to their multiplicity; then we set
$$\Phi(\sdim(V)):=\begin{cases}
\{\lambda_1,\dots,\lambda_N,0,0,\dots\}&\hbox{if }\#\sdim(V)=N\\
\{\lambda_1,\lambda_2,\dots\}&\hbox{if }\#\sdim(V)=+\infty.
\end{cases}$$
The function $g$ is a given function on $c_0(\R^-)$ with values in $]-\infty, +\infty]$.

Our main results are the existence of optimal potentials for the minimization problem \eqref{prob}, as precised in the following Theorems.

\begin{theo}\label{exth}
Let $F$ be a cost functional as in \eqref{cost}. We assume that the function $h$ satisfies 
\be\label{h0}
h(0)\ge0
\ee
and the following coercivity condition:
\be\label{lowerh}
\begin{split}
&h^-(t)\le M+c|t|^{p-d/2}\qquad\forall t<0\quad\text{if $d\ge3$} \\
&h^-(t)\le c|t|^{p-1}\qquad\forall t<0\quad\text{if $d=2$}
\end{split}
\ee
for suitable positive constants $M,c$ with $c<k/L_{p,d}$. Then the minimization problem
$$\min\Big\{F(V)\ :\ \spt V\subset K\Big\}$$
admits a solution.
\end{theo}

\begin{theo}\label{exth2}
Let $F$ be a cost functional as in \eqref{cost2}. We assume that the function $g$ is lower semicontinuous on $c_0(\R^-)$ (i.e. for the componentwise convergence) and satisfies the following coercivity condition:
\be\label{lowerh2}
g^-(\lambda)\le M+c|\lambda_1|^{p-d/2}\qquad\forall\lambda\in c_0(\R^-)
\ee
for suitable positive constants $M,c$ with $c<k/K_{p,d}$. Then the minimization problem
$$\min\Big\{F(V)\ :\ \spt V\subset K\Big\}$$
admits a solution provided the infimum is finite.
\end{theo}

\begin{rema}
We stress that in the definition \eqref{cost} of the cost functional $F$, the multiplicity $m(\lambda)$ appears. However it is easy to check that Theorem \ref{exth} and Theorem \ref{exth2} still hold if that coefficient $m(\lambda)$ is removed, providing we assume the sub-additivity of the function $h$ (i.e. $h(s+t)\le h(s) + h(t)$). On the other hand the assumption \eqref{h0} will be important for the existence issue in order to penalize negative eigenvalues close to $0$. Note that for $d\ge 3$, thanks to the Cwikel-Lieb-Rosenbljum bound \eqref{CLR}, the sum in \eqref{cost} is a finite sum, which is not true for $d=2$. 
\end{rema}

\begin{rema}\label{variant} 
By the same proof, the existence of optimal potentials obtained in Theorems \ref{exth} and \ref{exth2} still holds for cost functionals $F$ of the form
$$\sum_{\lambda\in\sdim(V)}m_V(\lambda)\,h(\lambda)+G(V)\quad\text{or}\quad g(\Phi(\sdim(V)))+G(V)\;,$$
where $G$ is any weakly lower semicontinuous functional in $L^p$ such that $G(V)\ge k\int|V|^p\,dx$. This includes for instance the case of constrained optimization problems of the kind
$$\min\Big\{F(V) \ :\ \spt V\subset K,\ |V|\le1\Big\}\;,$$
where we may drop the coercivity assumptions \eqref{lowerh} or \eqref{lowerh2}.
\end{rema}

\begin{exam}\label{setE}
Let us consider a compact set $E\subset \, ]-\infty,0[$ and the function
$$h(t)=-1_E(t)=\begin{cases}
-1&\hbox{if }x\in E\\
0&\hbox{if }x\notin E.
\end{cases}$$
The function $h$ above satisfies the assumptions of the existence Theorem \ref{exth}, and therefore according to Remark \ref{variant} the optimization problem
$$\max\Big\{\sum_{\lambda\in E\cap\sdim(V)}m_V(\lambda)\ :\ \spt V\subset K,\ |V|\le1\Big\}$$
admits a solution. This solution is then a potential $V$ that, among the ones supported by $K$ and with values in $[-1,1]$, has the maximum number of negative discrete eigenvalues in $E$, counted with their multiplicity.
\end{exam}

\begin{exam}\label{LTex}
Consider now a number $p>d/2$ and the function
$$h(t)=-|t|^{p-d/2}.$$
The function $h$ above satisfies the assumptions of the existence Theorem \ref{exth}, and therefore, also using the Remark \ref{variant}, the optimization problem
$$\max\Big\{\sum_{\lambda\in\sdim(V)}m_V(\lambda)|\lambda|^{p-d/2}\ :\ \spt V\subset K,\ V\le0,\ \int|V|^p\,dx\le1\Big\}$$
admits a solution. Notice that this provides, among negative potentials $V$ supported by $K$, the best potential for the Lieb-Thirring inequality \eqref{liethi}.
\end{exam}

\begin{exam}\label{Kex}
Consider a fixed natural number $N$ and a lower semicontinuous function $g:\R^N\to]-\infty,+\infty]$. For instance we may take
$$g(\lambda)=\lambda_j$$
in which we look for the lowest possible $j$-th negative eigenvalue, or
$$g(\lambda)=\lambda_1-\lambda_2$$
where we look for the maximal gap between $\lambda_2$ and $\lambda_1$ (under the convention that we take $\lambda_2=0$ whenever $\lambda_1$ is the only element of $\sdim(V)$). By the existence Theorem \ref{exth2} and Remark \ref{variant}, we deduce that the optimization problem
$$\min\Big\{g\big(\lambda_1(V),\dots,\lambda_N(V)\big)\ :\ \spt V\subset K,\ -1\le V\le0\Big\}$$
admits a solution. 
\end{exam}

\section{Proof of the results}\label{sproo}

We start by two useful lemmas.

\begin{lemm}\label{coercive} (Coercivity)
Let $p>d/2$ and let $h$ be a function satisfying the assumption \eqref{lowerh}. Then the functional $F$ in \eqref{cost}is coercive in $L^p$.
\end{lemm}

\begin{proof}
Let $(V_n)$ be such that $F(V_n) \le C$. Then, by Lieb-Thirring inequality
\[\begin{split}
C\ge&F(V_n)\ge k\int|V_n|^p\,dx-\sum_{\lambda\in\sdim(V_n)}m_{V_n}(\lambda)h^-(\lambda)\\
\ge&k\int|V_n|^p\,dx-\sum_{\lambda\in\sdim(V_n)}m_{V_n}(\lambda)(M+c|\lambda|^{p-d/2})\\
\ge&(k-cL_{p,d})\int|V_n|^p\,dx-M\#\sdim(V_n)
\end{split}\]
being $M=0$ if $d=2$. The conclusion is straightforward if $d=2$ whereas, if $d\ge3$, it follows from the CLR inequality \eqref{CLR} with exponent $q=d/2$.
\end{proof}


\begin{lemm}\label{strongres}
Let $V_n$ be a sequence of potentials converging to a potential $V$ weakly in $L^p$ with $p>d/2$. Then we have $R_n\to R$ strongly in $L^2$, where $R_n$ and $R$ are the resolvent operators corresponding to $V_n$ and $V$ respectively.
\end{lemm}

\begin{proof}
By the Lieb-Thirring inequality \eqref{liethi} all the negative eigenvalues of $-\Delta+V_n$ and of $-\Delta+V$ are uniformly bounded from below; let us take $k>0$ such that
$$k+\lambda_1(V_n)\ge1\qquad\forall n\in\N.$$
If $f\in L^2$ let us denote by $u_n,u\in H^1_0$ the solutions of
\be\label{resolveq}
-\Delta u_n+(V_n+k)u_n=f,\qquad-\Delta u+(V+k)u=f.
\ee

Since $k+\lambda_1(V_n)\ge1$ we have
$$\int u_n^2\,dx\le\int|\nabla u_n|^2+(V_n+k)u_n^2\,dx=\int fu_n\,dx$$
from which we deduce that $u_n$ is bounded in $L^2$. We show now that $u_n$ is bounded in $H^1$. By the equality above we obtain
\be\label{est1}
\int|\nabla u_n|^2\,dx\le C+\int V_n^-u_n^2\,dx.
\ee
We fix now a constant $M>0$; we have
\be\label{est2}
\begin{split}
\int V_n^-u_n^2\,dx&\le M\int_{\{V_n^-\le M\}}u_n^2\,dx+\int_{\{V_n^->M\}}u_n^2\,dx\\
&\le CM+\left(\int|u_n|^{2^*}dx\right)^{2/2^*}\left(\int_{\{V_n^->M\}}|V_n^-|^{d/2}\right)^{2/d}\\
&\le CM+C\left(\int|\nabla u_n|^2\,dx\right)\|V_n^-\|_{L^p}|\{V_n^->M\}|^{(2p-d)/pd}.
\end{split}
\ee
Since
$$|\{V_n^->M\}|\le\int_{\{V_n^->M\}}\frac{V_n^-}{M}\,dx\le\frac1M\|V_n^-\|_{L^p}|\{V_n^->M\}|^{(p-1)/p},$$
we obtain
$$|\{V_n^->M\}|\le CM^{-p},$$
and from \eqref{est2}
$$\int V_n^-u_n^2\,dx\le CM+\frac{C}{M^{(2p-d)/d}}\int|\nabla u_n|^2\,dx.$$
Taking $M$ such that $M^{(2p-d)/d}=2C$, from \eqref{est1} we deduce that $\int|\nabla u_n|^2\,dx$ is bounded.

Then we have that $u_n$ converges to $u$ in $L^2_{loc}$. Moreover, we can deduce from H\"older inequality that
\be\label{weakstrong}
\lim_{n\to\infty}\int V_n|u_n|^2\,dx=\int V|u|^2\,dx.
\ee
Indeed $|u_n|^2\to|u|^2$ a.e. and by Sobolev inequality
$$\int|u_n|^{2p'+\eps}\,dx=
\int|u_n|^{2^*}\,dx\le C\left(\int|\nabla u_n|^2\right)^{2^*/2},$$
being $\eps = 2^* -2 p' $. Then it follows from Vitali's convergence Theorem that $|u_n|^2 \to |u|^2$ strongly in $L^{p'}(K)$
and so \eqref{weakstrong} follows.

To finish the prove we need only to check that $u_n \to u$ strongly in $L^2$. By \eqref{resolveq}, we have
\[\begin{split}
&\int|\nabla u_n|^2\,dx+\int(V_n+k)|u_n|^2\,dx=\int fu_n\,dx\;,\\
&\int|\nabla u|^2 + \int (V + k)|u|^2\,dx=\int fu\,dx\;.
\end{split}\]
By the weak $L^2$ lower semicontinuity of the $H^1$-norm and \eqref{weakstrong}and recalling that $k>0$, we deduce that $\limsup_n \int |u_n|^2 \le \int |u|^2$
hence the conclusion. 
The strong convergence of resolvents follows by the classical argument that it is enough to check it for only one value outside the spectra.
\end{proof}

We are now in a position to prove the result of Theorem \ref{exth}.

\begin{proof}[Proof of Theorem \ref{exth}]We divide the proof in three steps.

\medskip

{\bf Step 1. }Consider a minimizing sequence $(V_n)$; thanks to the coercivity Lemma \ref{coercive}, the sequence $(V_n)$ is bounded in $L^p$ and so we may assume, up to extracting a subsequence, that it converges weakly to some function $V\in L^p$.

\bigskip

{\bf Step 2. }Since $V_n$ converges weakly to $V$, by Lemma \ref{strongres} we have the strong convergence of resolvent operators and hence that of the principal eigenvalues $\lambda_1(V_n)\to\lambda_1(V)$. In addition, the spectral measures $E_n$ related to the self-adjoint operators $-\Delta+V_n$ weakly converge to the spectral measure $E$ of $-\Delta+V$ (see for instance \cite{conca}). In other words, for every $\lambda$ which does not belong to the spectrum of $-\Delta+V$ we have
\be\label{weakmeas1}
\langle E_n(\lambda)\phi,\psi\rangle\to\langle E(\lambda)\phi,\psi\rangle\qquad\forall\phi,\psi\in L^2.
\ee
Since $\sdim(V)$ is finite, we have
\be\label{weakmeas2}
dE_n\big|_{]-\infty,\lambda[}=\sum_{t\in\sdim(V_n)}\delta_t\,P_{X_n(t)}
\ee
where $\delta_t$ is the Dirac mass at $t$ and $P_{X_n(t)}$ is the orthogonal projector on the finite dimensional eigenspace $X_n(t)$ associated to the eigenvalue $t$. From \eqref{weakmeas1} and \eqref{weakmeas2} we deduce that for any $\lambda<0$, with $\lambda\notin\sdim(V)$, $dE_n\big|_{]-\infty,\lambda[}\to dE\big|_{]-\infty,\lambda[}$ weakly in the sense of operators of finite rank and hence strongly. In particular, taking the trace on both sides, we deduce that
\be\label{spectralCV}
\mu_n\big|_{]-\infty,\lambda[}\to\mu\big|_{]-\infty,\lambda[}
\quad\forall\lambda\in\R^-\setminus\sdim(V)\;,
\ee
where the convergence is intended in the weak* convergence of measures and the nonnegative measures $\mu_n,\mu$ are defined by
\be\label{defmu}
\mu_n:=\sum_{t\in\sdim(V_n)}m_{V_n}(t)\,\delta_t\;,\qquad\mu:=\sum_{t\in\sdim(V)}m_V(t)\,\delta_t\;.
\ee

\bigskip

{\bf Step 3. }With the notations introduced in \eqref{defmu}, we may write
$$F(V_n)=\int h(t) \, d\mu_n(t) + k \int |V_n|^P \, dx,\qquad F(V) = \int h(t) \, d\mu(t) + k \int |V|^P \, dx\ .$$
In view of Steps 1 and 2, and recalling that $V_n\to V$ weakly in $L^p$, the existence of an optimal potential will be achieved as soon as 
we show the lower semicontinuity of $F$ which reduces to the inequality 
\be\label{last}
\liminf_n \int h(t)\,d\mu_n(t)\ge\int h(t)\,d\mu(t)\;.
\ee

We start with the case $d\ge3$. By \eqref{CLR}, it holds $\int d\mu_n\le C$ for a suitable constant $C$. Let $\eps>0$ be such that $-\eps\notin\sdim(V)$. Then
$$\int h(t)\,d\mu_n(t)
\ge\int_{]-\infty,-\eps[}h(t)\,d\mu_n-C\sup_{[-\eps,0]}h^-\;.$$
By \eqref{spectralCV} and by the lower semicontinuity of $h$ we obtain
$$\liminf_n\int h(t)\,d\mu_n(t)\ge\int_{]-\infty,-\eps[}h(t)\,d\mu-C\sup_{[-\eps,0]}h^-.$$
The conclusion \eqref{last} follows by the assumption $h(0)\ge0$ letting $\eps\to 0$. 

Let us now consider the case $d=2$ in which the measures $\mu_n$ can be unbounded in the vicinity of zero. By the assumption \eqref{lowerh}, we have
$$\int h(t)\,d\mu_n(t)\ge
\int_{]-\infty,-\eps[}h(t)\,d\mu_n-c\int_{]-\eps,0[}|t|^{p-1}\,d\mu_n\;.$$
Let $r$ such that $0<r<p-1$. Thanks to the Lieb-Thirring inequality \eqref{liethi} with exponent $p-r$, we have
$$\int_{]-\eps,0[}|t|^{p-1}\,d\mu_n\le\eps^r\int_{]-\eps,0[}|t|^{p-1-r}\,d\mu_n
\le\eps^r\,L_{p-r,2}\,\int|V_n^-|^{p-r}\,dx\le C\,\eps^r\;,$$
and similarly for $\mu$. Therefore as $n\to\infty$, we obtain
$$\liminf_n\int h(t)\,d\mu_n(t)
\ge\int_{]-\infty,-\eps[}h(t)\,d\mu-cC\eps^r
\ge\int h(t)\,d\mu-2cC\eps^r\;,$$
thus the conclusion \eqref{last} as $\eps\to 0$.
\end{proof}

\begin{proof}[Proof of Theorem \ref{exth2}] The proof follows the same scheme as the one of Theorem \ref{exth}. The coercivity of $F$ can be obtained as in Lemma \ref{coercive} using the inequality \eqref{Keller}. Step 2 remains unchanged and so the only difference is in Step 3. It is enough to observe that, thanks to the convergence of resolvents of spectral measures proved in Step 2, we have the convergence $ \Phi(\sdim(V_n) \to \Phi(\sdim(V)$ in $c_0(\R^-)$ hence the conclusion by the lower semicontinuity of $g$.
\end{proof}

\bigskip

\ack This work started during a visit of the second author at IMATH of University of Toulon. A part of this paper was written during a visit of the authors at the Johann Radon Institute for Computational and Applied Mathematics (RICAM) of Linz. The authors gratefully acknowledge both Institutes for the excellent working atmosphere provided. The second author is member of the Gruppo Nazionale per l'Analisi Matematica, la Probabilit\`a e le loro Applicazioni (GNAMPA) of the Istituto Nazionale di Alta Matematica (INdAM), and his work is part of the project 2010A2TFX2 {\it``Calcolo delle Variazioni''} funded by the Italian Ministry of Research and University.


\bigskip
{\small\noindent
Guy Bouchitt\'e:
Laboratoire IMATH,
Universit\'e de Toulon\\
BP 20132,
83957 La Garde Cedex - FRANCE\\
{\tt bouchitte@univ-tln.fr}\\
{\tt https://sites.google.com/site/gbouchitte/home}

\bigskip\noindent
Giuseppe Buttazzo:
Dipartimento di Matematica,
Universit\`a di Pisa\\
Largo B. Pontecorvo 5,
56127 Pisa - ITALY\\
{\tt buttazzo@dm.unipi.it}\\
{\tt http://www.dm.unipi.it/pages/buttazzo/}

\end{document}